\documentclass[preprint,12pt]{elsarticle}
\usepackage[latin1]{inputenc}
\usepackage{amsmath,amsthm,amssymb}
\usepackage{amsfonts}
\usepackage{amsmath,amsthm,amssymb,amscd}
\usepackage{latexsym}
\usepackage{color}
\usepackage{graphicx}
\usepackage{mathrsfs}

\makeatletter \@addtoreset{equation}{section} \makeatother

\newtheorem{theorem}{Theorem}[section]

\newtheorem{lemma}{Lemma}[section]

\journal{Journal}

\begin{document}

\begin{frontmatter}

\title{Ground state sign-changing solutions for a class of  nonlinear fractional Schr\"{o}dinger-Poisson system in $\mathbb{R}^{3}$}

\author{Chao Ji}
\address{Department of Mathematics, East China University of Science and Technology,  Shanghai, 200237, China}
\ead{jichao@ecust.edu.cn}

\begin{abstract}

In this paper, we are concerned with the existence of the least energy sign-changing solutions for the following fractional Schr\"{o}dinger-Poisson system:
\begin{align*}
 \left\{
\begin{aligned}
&(-\Delta)^{s} u+V(x)u+\lambda\phi(x)u=f(x, u),\quad &\text{in}\, \ \mathbb{R}^{3},\\
&(-\Delta)^{t}\phi=u^{2},& \text{in}\,\ \mathbb{R}^{3},
\end{aligned}
\right.
\end{align*}
where $\lambda\in \mathbb{R}^{+}$ is a parameter, $s, t\in (0, 1)$ and $4s+2t>3$, $(-\Delta)^{s}$ stands for the fractional Laplacian. By constraint variational method and quantitative
 deformation lemma, we prove that the above problem has one least energy sign-changing solution. Moreover, for any $\lambda>0$, we show that the energy of the least energy sign-changing solutions is strictly larger than two times the ground state energy.
 Finally, we consider $\lambda$ as a parameter and study the convergence property of  the least energy sign-changing solutions as $\lambda\searrow 0$.
\end{abstract}


\end{frontmatter}
\section{Introduction}\label{intro}
In this article, we are interested in the existence, energy property of the least energy sign-changing
solution $u_{\lambda}$ and a convergence property of $u_{\lambda}$ as $\lambda\searrow 0$ for the nonlinear  fractional Schr\"{o}dinger-Poisson system
\begin{align}
 \left\{
\begin{aligned}
&(-\Delta)^{s} u+V(x)u+\lambda\phi(x)u=f(x, u),\quad &\text{in}\, \ \mathbb{R}^{3},\\
&(-\Delta)^{t}\phi=u^{2},& \text{in}\,\ \mathbb{R}^{3}.
\end{aligned}
\right.
\end{align}
where $\lambda>0$ is a parameter, $s, t\in (0, 1)$ and $4s+2t>3$, $(-\Delta)^{s}$ stands for the fractional Laplacian
and the potential $V(x)$ satisfies the following assumptions:\\
\textbf{$(V_{1})$} $V\in C(\mathbb{R}^{3})$ satisfies $\underset{x\in \mathbb{R}^{3}}{\inf}V(x)\geq V_{0}>0$, where $V_{0}$ is a positive constant;\\
\textbf{$(V_{2})$} There exists $h>0$ such that $\underset{\vert y\vert\rightarrow\infty}{\lim}\text{meas}(\{x\in B_{h}(y): V(x)\leq c\})=0$ for any $c>0$;\\
where $B_{h}(y)$ denotes an open ball of $\mathbb{R}^{3}$ centered at $y$ with radius $h>0$, and $\text{meas}(A)$ denotes the Lebesgue measure of set $A$. Condition $(V_{2})$, which is weaker than the coercivity assumption: $V(x)\rightarrow \infty$ as $\vert x\vert\rightarrow \infty$, was originally introduced
by Bartsch and Wang in  \cite{rBW} to overcome the lack of compactness for the local elliptic equations and then was used by Pucci, Xia and Zhang in \cite{rPXZ} for the fractional Schr\"{o}dinger-Kirchhoff type equations. Moreover,
on the nonlinearity $f$ we assume that
\begin{itemize}
\item[$(f_1)$] $f: \mathbb{R}^{3}\times \mathbb{R}\rightarrow \mathbb{R}$ is a Carath\'{e}odory function and $f(x, u)=o(\vert u\vert)$ as $u\rightarrow 0$ for $x\in\mathbb{R}^{3}$ uniformly.

\item[$(f_2)$] For some $1<p<2_{s}^{*}-1$, there exits $C>0$ such that
\begin{equation*}
\vert f(x, u)\vert\leq C(1+\vert u\vert^{p}),
\end{equation*}
where $2_{s}^{*}=\frac{6}{3-2s}$.

\item[$(f_3)$] $\underset{t\rightarrow \infty}{\lim}\frac{F(x, t)}{t^{4}}=+\infty$, where $ F(t)=\int_{0}^{t}f(s)ds$.

\item[$(f_4)$] $\frac{f(x, t)}{\vert t\vert^{3}}$ is an increasing function of $t$ on $\mathbb{R}\setminus\{0\}$ for a.e. $x\in\mathbb{R}^{3}$.
\end{itemize}

When $s=t=1$, the system (1.1) reduces to the following Schr\"{o}dinger-Poisson system
\begin{align*}
 \left\{
\begin{aligned}
&-\Delta u+V(x)u+\lambda\phi(x)u=f(x, u),\quad &\text{in}\, \ \mathbb{R}^{3},\\
&-\Delta\phi=u^{2},& \text{in}\,\ \mathbb{R}^{3}.
\end{aligned}
\right.
\end{align*}
This kind of system has a strong physical meaning. For instance, they appear in quantum mechanics models (\cite{rBBL, rCL}), and in semiconductor theory(\cite{rBF1, rBF2}). For the research of Schr\"{o}dinger-Poisson system, we may refer to \cite{rIV, rJZ, rLPY, rR, rWZ}.

In recent years, there has been a great deal work dealing with the nonlinear equations or systems involving fractional Laplacian operators, which arise in
fractional quantum mechanics \cite{rL1, rL2}, physics and chemistry \cite{rMK}, obstacle problems \cite{rSi}, optimization and finance \cite{rCT} and so son. In the remarkable work of Caffarelli and Silvestre \cite{rCS}, the authors express this nonlocal operator $(-\Delta)^{s}$ as a Dirichlet-Neumann map for a certain elliptic boundary value problem with local differential operators defined on the upper boundary. This technique is a valid tool to deal with the equations involving fractional operators in the respects of regularity and variational methods. For some results on the fractional differential equations, we refer to \cite{rCW, rMR, rPXZ, rZZX, rZZR}. Recently, Using the method in \cite{rCS} and variational method, in \cite{rT}, Teng studied the ground state for the fractional Schr\"{o}dinger-Poisson system with critical Sobolev exponent.  To the best of our knowledge, there are few papers which
considered the least energy sign-changing solutions of system (1.1). In \cite{rSh}, Combining constraint variational methods and quantitative deformation lemma, Shuai firstly studied the least energy sign-changing solutions for a class of Kirchhoff problems in bounded domain, where a stronger condition that $f\in C^{1}$ was assumed.
In virtue of the fractional operator and Poisson equation are included in (1.1), our problem is more complicated and difficult.

Now we recall the fractional Sobolev spaces. We firstly define the homogeneous fractional Sobolev space $\mathcal{D}^{\alpha, 2}(\mathbb{R}^{3})$ as follows\\
\begin{equation*} \label{space}
 \mathcal{D}^{\alpha, 2}(\mathbb{R}^{3})=\Big\{u\in L^{2_{\alpha}^{*}}(\mathbb{R}^{3}): \frac{\vert u(x)-u(y)\vert}{\vert x-y\vert^{\frac{3}{2}+\alpha}}\in L^{2}(\mathbb{R}^{3}\times \mathbb{R}^{3})\Big\}
\end{equation*}
which is the completion of $\mathcal{C}_{0}^{\infty}(\mathbb{R}^{3})$ under the norm\\
\begin{equation*} \label{space}
 \Vert u\Vert_{ \mathcal{D}^{\alpha, 2}(\mathbb{R}^{3})}=\Vert (-\Delta)^{\frac{\alpha}{2}} u\Vert_{L^{2}(\mathbb{R}^{3})}=\Big(\iint_{\mathbb{R}^{3}\times\mathbb{R}^{3}}\frac{\vert u(x)-u(y)\vert^{2}}{\vert x-y\vert^{3+2\alpha}}dxdy\Big)^{\frac{1}{2}}.
\end{equation*}

The embedding $ \mathcal{D}^{\alpha, 2}(\mathbb{R}^{3})\hookrightarrow L^{2_{\alpha}^{*}}$ is continuous and for any $\alpha\in (0, 1)$, there
exists a best constant $\mathcal{S}_{\alpha}>0$ such that
\begin{equation} \label{space}
\mathcal{S}_{\alpha}=\underset{u\in \mathcal{D}^{\alpha, 2}(\mathbb{R}^{3})}{\inf}=\frac{\int_{\mathbb{R}^{3}} \vert (-\Delta)^{\frac{\alpha}{2}}u\vert^{2}dx}{\Big(\int_{\mathbb{R}^{3}} \vert u(x)\vert^{2_{\alpha}^{*}}dx\Big)^{\frac{2}{2_{\alpha}^{*}}}}.
\end{equation}

The fractional Sobolev space $H^{\alpha}(\mathbb{R}^{3})$ is defined by
\begin{equation*} \label{space}
 H^\alpha(\mathbb{R}^{3})=\Big\{u\in L^{2}(\mathbb{R}^{3}): \frac{\vert u(x)-u(y)\vert}{\vert x-y\vert^{\frac{3}{2}+\alpha}}\in L^{2}(\mathbb{R}^{3}\times \mathbb{R}^{3})\Big\},
\end{equation*}
endowed with the norm
\begin{equation*} \label{space}
 \Vert u\Vert_{H^\alpha(\mathbb{R}^{3})}=\Big(\iint_{\mathbb{R}^{3}\times\mathbb{R}^{3}}\frac{\vert u(x)-u(y)\vert^{2}}{\vert x-y\vert^{3+2\alpha}}dxdy+\int_{\mathbb{R}^{3}} \vert u\vert^{2}dx\Big)^{\frac{1}{2}}.
\end{equation*}

In this paper, we denote the fractional Sobolev space for (1.1) by
\begin{equation*} \label{space}
H= \Big\{u\in H^s(\mathbb{R}^{3}): \int_{\mathbb{R}^{3}}V(x)u^{2}dx< \infty\Big\},
\end{equation*}
with the norm
\begin{equation*} \label{space}
 \Vert u\Vert=\Big(\iint_{\mathbb{R}^{3}\times\mathbb{R}^{3}}\frac{\vert u(x)-u(y)\vert^{2}}{\vert x-y\vert^{3+2s}}dxdy+\int_{\mathbb{R}^{3}}  V(x)u^{2}dx\Big)^{\frac{1}{2}}.
\end{equation*}

In the sequel, we need the following embedding lemma which is a special case of Lemma 1 in \cite{rPXZ}, so we omit its proof.

\begin{lemma}
$(i)$ Suppose that $(V_{1})$ holds. Let $q\in [2, 2_{s}^{*}]$, then the embeddings
\begin{equation*}
H\hookrightarrow H^{s}(\mathbb{R}^{3})\hookrightarrow L^{q}(\mathbb{R}^{3})
\end{equation*}
are continuous, with $\min\{1, V_{0}\}\Vert u\Vert_{H^s(\mathbb{R}^{3})}^{2}\leq \Vert u\Vert^{2}$ for all $u\in H$. In particular, there exists a constant
$C_{q}>0$ such that
\begin{equation*}
\Vert u\Vert_{L^q(\mathbb{R}^{3})}\leq C_{q}\Vert u\Vert\quad \text{for all}\,\, u\in H.
\end{equation*}
Moreover, if $q\in[1, 2_{s}^{*})$, then the embedding $H\hookrightarrow\hookrightarrow L^{q}(B_{R})$ is compact for any $R>0$.\\
$(ii)$ Suppose that $(V_{1})-(V_{2})$ hold. Let $q\in [2, 2_{s}^{*})$ be fixed and $\{u_{n}\}$ be a bounded sequence in $H$, then there exists $u\in H\cap L^{q}(\mathbb{R}^{N})$ such that, up to a subsequence,
\begin{equation*}
 u_{n}\rightarrow u \quad \text{strongly in}\,\, L^q(\mathbb{R}^{3})\, \text{as}\,\, n\rightarrow \infty.
\end{equation*}
\end{lemma}

Assume that $s, t\in (0, 1)$, if $4s+2t\geq 3$, there holds $2\leq \frac{12}{3+2t} \leq \frac{6}{3-2s}$ and thus $H\hookrightarrow L^{\frac{12}{3+2t}}(\mathbb{R}^{3})$ by Lemma 1.1. For $u\in H$, the linear functional $\mathcal{L}_{u}:\mathcal{D}^{t, 2}(\mathbb{R}^{3})\rightarrow \mathbb{R}$ is defined by
\begin{equation*} \label{space}
 \mathcal{L}_{u}(v)=\int_{\mathbb{R}^{3}} u^{2}vdx,
\end{equation*}
the H\"{o}lder's inequality and (1.2) implies that
\begin{equation*} \label{space}
 \vert \mathcal{L}_{u}(v)\vert \leq \Big(\int_{\mathbb{R}^{3}} \vert u(x)\vert^{\frac{12}{3+2t}}dx\Big)^{\frac{3+2t}{6}}\Big(\int_{\mathbb{R}^{3}} \vert v(x)\vert^{2_{t}^{*}}dx\Big)^{\frac{1}{2_{t}^{*}}}\leq C\Vert u\Vert^{2}\Vert v\Vert_{ \mathcal{D}^{t, 2}(\mathbb{R}^{3})}.
\end{equation*}
 By the Lax-Milgram theorem, there exists a unique $\phi_{u}^{t}\in \mathcal{D}^{t, 2}(\mathbb{R}^{3})$ such that
\begin{equation*}
\int_{\mathbb{R}^{3}}(-\Delta)^{\frac{t}{2}}u(-\Delta)^{\frac{t}{2}}vdx=\int_{\mathbb{R}^{3}}u^{2}vdx,\quad  \forall v\in \mathcal{D}^{t, 2}(\mathbb{R}^{3}),
\end{equation*}
 that is $\phi_{u}^{t}$ is the weak solution of
 \begin{equation*}
 (-\Delta)^{t}\phi_{u}^{t}=u^{2},\quad x\in\mathbb{R}^{3}
 \end{equation*}
 and the representation formula holds
\begin{equation*}
\phi_{u}^{t}(x)=c_{t}\int_{\mathbb{R}^{3}} \frac{u^{2}(y)}{\vert x-y\vert^{3-2t}}dy,\quad  x\in\mathbb{R}^{3},
\end{equation*}
 which is called t-Riesz potential, where
 \begin{equation}
c_{t}=\pi^{-\frac{3}{2}}2^{-2t}\frac{\Gamma(3-2t)}{\Gamma(t)}.
\end{equation}

In the sequel, we often omit the constant $c_{t}$ for convenience in (1.3). The properties of the function $\phi_{u}^{t}$  are given as follows.

\begin{lemma}[\cite{rT}]
If $4s+2t\geq 3$, then for any  $u\in H^s(\mathbb{R}^{3})$, we have\\
$(1)$$\phi_{u}^{t}: H^s(\mathbb{R}^{3})\rightarrow \mathcal{D}^{t, 2}(\mathbb{R}^{3})$ is continuous and maps bounded sets into bounded maps;\\
$(2)$$\int_{\mathbb{R}^{3}}\phi_{u}^{t}{u}^{2}dx\leq \mathcal{S}_{t}^{2}\Vert u\Vert^{4}_{L^{\frac{12}{3+2t}}}$;\\
$(3)$$\phi_{\tau u}^{t}=\tau^{2}\phi_{u}^{t}$ for all $\tau\in \mathbb{R}$, $\phi_{u(\cdot+y)}^{t}=\phi_{u}^{t}(x+y)$;\\
$(4)$$\phi_{u_{\theta}}=\theta^{2s}(\phi_{u}^{t})_{\theta}$ for all $\theta>0$, where $u_{\theta}=u(\frac{\cdot}{\theta});$\\
$(5)$If $u_{n}\rightharpoonup u$ in $H^s(\mathbb{R}^{3})$ then $\phi_{u_{n}}^{t}\rightharpoonup \phi_{u}^{t}$ in $\mathcal{D}^{t, 2}(\mathbb{R}^{3})$;\\
$(6)$ If $u_{n}\rightarrow u$ in $H^s(\mathbb{R}^{3})$ then $\phi_{u_{n}}^{t}\rightarrow \phi_{u}^{t}$ in $\mathcal{D}^{t, 2}(\mathbb{R}^{3})$ and
$\int_{\mathbb{R}^{3}}\phi_{u_{n}}^{t}{u_{n}}^{2}dx\rightarrow \int_{\mathbb{R}^{3}}\phi_{u}^{t}{u}^{2}dx$.
\end{lemma}

If we substitute $\phi_{u}^{t}$ in (1.1), it leads to the following fractional
Schr\"{o}dinger equation
\begin{equation}
(-\Delta)^{s} u+V(x)u+\lambda\phi_{u}^{t}u=f(x, u),\quad \text{in}\, \ \mathbb{R}^{3},\\
\end{equation}
whose solutions are the critical points of the functional $I_{\lambda}: H\rightarrow \mathbb{R}$ defined by
\begin{equation*}
I_{\lambda}(u)=\frac{1}{2}\int_{\mathbb{R}^{3}}(\vert (-\Delta)^{\frac{s}{2}}u\vert^{2}+V(x)  u^{2})dx+\frac{\lambda}{4}\int_{\mathbb{R}^{3}}\phi_{u}^{t}u^{2}dx-\int_{\mathbb{R}^{3}}F(x, u)dx
\end{equation*}
where $F(x, u)=\int_{0}^{u}f(x, r)dr$. The functional  $I_{\lambda}\in C^{1}(H, \mathbb{R})$  and for any $v\in H$
\begin{equation*}
\langle I_{\lambda}'(u), v\rangle=\int_{\mathbb{R}^{3}}\Big((-\Delta)^{\frac{s}{2}}u(-\Delta)^{\frac{s}{2}}v+V(x)uv\Big)dx+\lambda\int_{\mathbb{R}^{3}}\phi_{u}^{t}uv dx-\int_{\mathbb{R}^{3}}f(x,u)v dx.
\end{equation*}
We call $u$ a least energy sign-changing solution to problem (1.1) if $u$ is a solution of problem (1.4) with $u^{\pm}\neq 0$ and
\begin{equation*}
I_{\lambda}(u)=\inf \{I_{\lambda}(v)\mid v^{\pm}\neq 0, I'_{\lambda}(v)=0\},
\end{equation*}
where $v^{+}=\max\{v(x), 0\}$ and $v^{-}=\min\{v(x), 0\}$.

For problem (1.4),  due to the effect of the nonlocal term $\phi_{u}^{t}$ and $(-\Delta)^{s} u$, that is
\begin{equation*}
\int_{\mathbb{R}^{3}}\Big((-\Delta)^{\frac{s}{2}}u^{+}(-\Delta)^{\frac{s}{2}}u^{-}dx>0\quad\text{and}\quad
\int_{\mathbb{R}^{3}}\phi_{u}^{t}u^{2}dx>\int_{\mathbb{R}^{3}}\phi_{u^{+}}^{t}{u^{+}}^{2}+\int_{\mathbb{R}^{3}}\phi_{u^{-}}^{t}{u^{-}}^{2}
\end{equation*}
for $u^{\pm}\neq 0$, a straightforward computation yields that
\begin{equation*}
I_{\lambda}(u)>I_{\lambda}(u^{+})+ I_{\lambda}(u^{-}),
\end{equation*}
\begin{equation*}
\langle I_{\lambda}'(u), u^{+}\rangle>\langle I_{\lambda}'(u^{+}), u^{+}\rangle,\quad \text{and}\quad \langle I_{\lambda}'(u), u^{-}\rangle>\langle I_{\lambda}'(u^{-}), u^{-}\rangle.
\end{equation*}
So the methods to obtain sign-changing solutions of the local problems and to estimate the energy of the sign-changing solutions seem not suitable for our nonlocal one (1.4). In order to get a sign-changing solution for problem (1.4),  we firstly try to seek a minimizer of the energy functional $I_{\lambda}$ over the following constraint:
\begin{equation*}
\mathcal{M}_{\lambda}=\{u\in H: u^{\pm}\neq 0, \langle I_{\lambda}'(u), u^{+}\rangle=\langle I_{\lambda}'(u), u^{-}\rangle=0\}
\end{equation*}
and then we show that the minimizer is a sign-changing solution of (1.4). To show that the minimizer of the constrained problem is a sign-changing solution, we will use the quantitative deformation lemma and degree theory.

The following are the main results of this paper.

\begin{theorem}\label{gt621}
Let $(f_{1})-(f_{4})$ and $(V_{1})-(V_{2})$ hold. Then for any $ \lambda>0$, problem (1.1) has a least energy sign-changing solution $u_{\lambda}$, which has precisely two
nodal domains.
\end{theorem}

 Let
\begin{equation}
\mathcal{N}_{\lambda}:=\{u\in H\setminus \{0\}:  \langle I_{\lambda}'(u), u\rangle=0\},
\end{equation}
and
\begin{equation}
c_{\lambda}:=\underset{u\in \mathcal{N}_{\lambda}}{\inf}I_{\lambda}(u).
\end{equation}
Let $u_{\lambda}\in H$ be a sign-changing solution of problem (1.4), it is clear from (1.7) and (1.8) that $u_{\lambda}^{\pm}\not\in \mathcal{N}_{\lambda}$.
\begin{theorem}\label{gt621}
Under the assumptions of Theorem 1.1, $c_{\lambda}>0$ is achieved and $I_{\lambda}(u_{\lambda})>2c_{\lambda}$, where
$u_{\lambda}$ is the least energy sign-changing solution obtained in Theorem 1.1. In particular, $c_{\lambda}>0$ is achieved
 either by a positive or a negative function.
\end{theorem}

It is clear that the energy of the sign-changing solution $u_{\lambda}$ obtained in Theorem1.1 depends on $\lambda$. As a by-product of this paper, we give a convergence property of $u_{\lambda}$  as $\lambda\searrow 0$, which reflects some relationship between $\lambda>0$ and $\lambda=0$ in problem (1.4).

\begin{theorem}\label{gt621}
If the assumptions of Theorem 1.1 hold, then for any sequence $\{\lambda_{n}\}$ with $\lambda_{n}\searrow 0$ as $n\rightarrow \infty$, there exists a subsequence, still denoted by $\{\lambda_{n}\}$, such that $u_{\lambda_{n}}\rightarrow u_{0}$ strongly in $H$ as $n\rightarrow \infty$, where $u_{0}$ is a least energy sign-changing solution of the problem
\begin{equation}
(-\Delta)^{s} u+V(x)u=f(x, u),\quad \text{in}\,\ \mathbb{R}^{3},
\end{equation}
which has precisely two
nodal domains.
\end{theorem}

This paper is organized as follows. In Section2, we
present some preliminary lemmas which are essential for this paper. In Section 3, we give the proofs of Theorems 1.1--1.3 respectively.

\section{Some Technical Lemmas}\label{prel}
We will use constraint minimization on $\mathcal{M}_{\lambda}$ to look for a critical point of $I_{\lambda}$. For this, we start with this section by claiming that the set $\mathcal{M}_{\lambda}$ is nonempty in $H$.

\begin{lemma}
Assume that $(f_{1})-(f_{4})$ and $(V_{1})$ hold, if $u\in H$ with $u^{\pm}\neq 0$, then there exists a unique pair $(\alpha_{u}, \beta_{u})\in \mathbb{R}_{+}\times \mathbb{R}_{+}$ such that $\alpha_{u}u^{+}+\beta_{u}u^{-}\in \mathcal{M}_{\lambda}$.
\end{lemma}
\begin{proof}
Fixed an $u\in H$ with $u^{\pm}\neq 0$. We first establish the existence of $\alpha_{u}$ and $\beta_{u}$. Let
\begin{align}
g_{1}(\alpha, \beta)&=\langle I_{\lambda}'(\alpha u^{+}+\beta u^{-}), \alpha u^{+}\rangle\nonumber\\
&=\int_{\mathbb{R}^{3}}(-\Delta)^{\frac{s}{2}}(\alpha u^{+}+\beta u^{-})(-\Delta)^{\frac{s}{2}}(\alpha u^{+})dx +\alpha^{2}\int_{\mathbb{R}^{3}}V(x){u^{+}}^{2}dx\nonumber\\
&\quad +\lambda \alpha^{2}\int_{\mathbb{R}^{3}}\phi_{\alpha u^{+}+\beta u^{-}}^{t}{u^{+}}^{2} dx-\int_{\mathbb{R}^{3}}f(x,\alpha u^{+})\alpha u^{+} dx\nonumber\\
&=\alpha^{2}\int_{\mathbb{R}^{3}}\vert(-\Delta)^{\frac{s}{2}}u^{+}\vert^{2} dx +\alpha\beta \int_{\mathbb{R}^{3}}(-\Delta)^{\frac{s}{2}}u^{+}(-\Delta)^{\frac{s}{2}}u^{-}dx\nonumber\\
&\quad +\alpha^{2}\int_{\mathbb{R}^{3}}V(x){u^{+}}^{2}dx+\lambda \alpha^{4}\int_{\mathbb{R}^{3}}\phi_{u^{+}}^{t}{u^{+}}^{2}dx+\lambda \alpha^{2}\beta^{2}\int_{\mathbb{R}^{3}}\phi_{u^{-}}^{t}{u^{+}}^{2}dx\nonumber\\
&\quad -\int_{\mathbb{R}^{3}}f(x,\alpha u^{+})\alpha u^{+}dx,
\end{align}
and
\begin{align*}
g_{2}(\alpha, \beta)&=\langle I_{\lambda}'(\alpha u^{+}+\beta u^{-}), \beta u^{-}\rangle\nonumber\\
&=\int_{\mathbb{R}^{3}}(-\Delta)^{\frac{s}{2}}(\alpha u^{+}+\beta u^{-})(-\Delta)^{\frac{s}{2}}(\beta u^{-})dx +\beta^{2}\int_{\mathbb{R}^{3}}V(x){u^{-}}^{2}dx\nonumber\\
&\quad +\lambda \beta^{2}\int_{\mathbb{R}^{3}}\phi_{\alpha u^{+}+\beta u^{-}}^{t}{u^{-}}^{2} dx-\int_{\mathbb{R}^{3}}f(x,\beta u^{-})\beta u^{-} dx\nonumber\\
\end{align*}
\begin{align}
&\quad =\beta^{2}\int_{\mathbb{R}^{3}}\vert(-\Delta)^{\frac{s}{2}}u^{-}\vert^{2} dx +\alpha\beta \int_{\mathbb{R}^{3}}(-\Delta)^{\frac{s}{2}}u^{+}(-\Delta)^{\frac{s}{2}}u^{-}dx\nonumber\\
&\quad\quad +\beta^{2}\int_{\mathbb{R}^{3}}V(x){u^{-}}^{2}dx+\lambda \beta^{4}\int_{\mathbb{R}^{3}}\phi_{u^{-}}^{t}{u^{-}}^{2}dx+\lambda \alpha^{2}\beta^{2}\int_{\mathbb{R}^{3}}\phi_{u^{+}}^{t}{u^{-}}^{2}dx\nonumber\\
&\quad\quad -\int_{\mathbb{R}^{3}}f(x,\beta u^{-})\beta u^{-}dx.
\end{align}
By $(f_{1})$ and $(f_{3})$, it is easy to see that $g_{1}(\alpha, \alpha)>0$ and $g_{2}(\alpha, \alpha)>0$ for $\alpha>0$ small and
$g_{1}(\beta, \beta)<0$ and $g_{2}(\beta, \beta)<0$ for $\beta>0$ large. Thus, there exist $0<r<R$ such that
\begin{equation}
g_{1}(r, r)>0, \quad g_{2}(r, r)>0,\quad g_{1}(R, R)<0,\quad g_{2}(R, R)<0.
\end{equation}
From (2.1), (2.2) and (2.3), we have
\begin{equation*}
g_{1}(r, \beta)>0,  g_{1}(\beta, R)<0\quad \forall \beta\in [r, R]
\end{equation*}
and
\begin{equation*}
g_{2}(\alpha, r)>0,  g_{2}(\alpha, R)<0\quad \forall \alpha\in [r, R].
\end{equation*}
By virtue of Miranda's Theorem\cite{rM}, there exists some point $(\alpha_{u}, \beta_{u})$ with $r<\alpha_{u}, \beta_{u}<R$ such that
$g_{1}(\alpha_{u}, \beta_{u})=g_{2}(\alpha_{u}, \beta_{u})=0$. So  $\alpha_{u}u^{+}+\beta_{u}u^{-}\in\mathcal{M}_{\lambda}$.\\

Now, we prove the uniqueness of the pair $(\alpha_{u}, \beta_{u})$ and consider two cases.\\
\textbf{Case 1.} $u\in \mathcal{M}_{\lambda}$.\\
If $u\in \mathcal{M}_{\lambda}$, then $u^{+}+u^{-}=u\in \mathcal{M}_{\lambda}$. It means that
\begin{equation*}
 \langle I_{\lambda}'(u), u^{+}\rangle=\langle I_{\lambda}'(u), u^{-}\rangle=0,
\end{equation*}
that is
\begin{align}
&\int_{\mathbb{R}^{3}}\vert(-\Delta)^{\frac{s}{2}}u^{+}\vert^{2} dx + \int_{\mathbb{R}^{3}}(-\Delta)^{\frac{s}{2}}u^{+}(-\Delta)^{\frac{s}{2}}u^{-}dx
+\int_{\mathbb{R}^{3}}V(x){u^{+}}^{2}dx\nonumber\\
&+\lambda \int_{\mathbb{R}^{3}}\phi_{u^{+}}^{t}{u^{+}}^{2}dx+\lambda \int_{\mathbb{R}^{3}}\phi_{u^{-}}^{t}{u^{+}}^{2}dx
=\int_{\mathbb{R}^{3}}f(x, u^{+}) u^{+}dx,
\end{align}
and
\begin{align}
&\int_{\mathbb{R}^{3}}\vert(-\Delta)^{\frac{s}{2}}u^{-}\vert^{2} dx +\int_{\mathbb{R}^{3}}(-\Delta)^{\frac{s}{2}}u^{+}(-\Delta)^{\frac{s}{2}}u^{-}dx
+\int_{\mathbb{R}^{3}}V(x){u^{-}}^{2}dx\nonumber\\
&+\lambda \int_{\mathbb{R}^{3}}\phi_{u^{-}}^{t}{u^{-}}^{2}dx+\lambda \int_{\mathbb{R}^{3}}\phi_{u^{+}}^{t}{u^{-}}^{2}dx
=\int_{\mathbb{R}^{3}}f(x, u^{-}) u^{-}dx.
\end{align}
We show that $(\alpha_{u}, \beta_{u})=(1, 1)$ is the unique pair of numbers such that $\alpha_{u}u^{+}+\beta_{u}u^{-}\in\mathcal{M}_{\lambda}$.\\
Assume that $(\tilde{\alpha}_{u}, \tilde{\beta}_{u})$ is another pair of numbers such that $\tilde{\alpha}_{u}u^{+}+\tilde{\beta}_{u}u^{-}\in\mathcal{M}_{\lambda}$. By the definition of $\mathcal{M}_{\lambda}$, we have
\begin{align}
&\tilde{\alpha}_{u}^{2}\int_{\mathbb{R}^{3}}\vert(-\Delta)^{\frac{s}{2}}u^{+}\vert^{2} dx +\tilde{\alpha}_{u}\tilde{\beta}_{u} \int_{\mathbb{R}^{3}}(-\Delta)^{\frac{s}{2}}u^{+}(-\Delta)^{\frac{s}{2}}u^{-}dx+\tilde{\alpha}_{u}^{2}\int_{\mathbb{R}^{3}}V(x){u^{+}}^{2}dx \nonumber\\
&\quad \quad +\lambda \tilde{\alpha}_{u}^{4}\int_{\mathbb{R}^{3}}\phi_{u^{+}}^{t}{u^{+}}^{2}dx+\lambda \tilde{\alpha}_{u}^{2}\tilde{\beta}_{u}^{2}\int_{\mathbb{R}^{3}}\phi_{u^{-}}^{t}{u^{+}}^{2}dx
=\int_{\mathbb{R}^{3}}f(x,\tilde{\alpha}_{u} u^{+})\tilde{\alpha}_{u} u^{+}dx,
\end{align}
and
\begin{align}
&\tilde{\beta}_{u}^{2}\int_{\mathbb{R}^{3}}\vert(-\Delta)^{\frac{s}{2}}u^{-}\vert^{2} dx +\tilde{\alpha}_{u}\tilde{\beta}_{u} \int_{\mathbb{R}^{3}}(-\Delta)^{\frac{s}{2}}u^{+}(-\Delta)^{\frac{s}{2}}u^{-}dx
+\tilde{\beta}_{u}\int_{\mathbb{R}^{3}}V(x){u^{-}}^{2}dx\nonumber\\
&\quad \quad+\lambda \tilde{\beta}_{u}^{4}\int_{\mathbb{R}^{3}}\phi_{u^{-}}^{t}{u^{-}}^{2}dx+\lambda \tilde{\alpha}_{u}^{2}\tilde{\beta}_{u}^{2}\int_{\mathbb{R}^{3}}\phi_{u^{+}}^{t}{u^{-}}^{2}dx=\int_{\mathbb{R}^{3}}f(x,\tilde{\beta}_{u} u^{-})\tilde{\beta }_{u}u^{-}dx.
\end{align}
Without loss of generality, we may assume that $0<\tilde{\alpha}_{u}\leq \tilde{\beta}_{u}$. Then, from (2.6), we have
\begin{align*}
&\tilde{\alpha}_{u}^{2}\int_{\mathbb{R}^{3}}\vert(-\Delta)^{\frac{s}{2}}u^{+}\vert^{2} dx +\tilde{\alpha}_{u}^{2} \int_{\mathbb{R}^{3}}(-\Delta)^{\frac{s}{2}}u^{+}(-\Delta)^{\frac{s}{2}}u^{-}dx+\tilde{\alpha}_{u}^{2}\int_{\mathbb{R}^{3}}V(x){u^{+}}^{2}dx\\
&\quad \quad+\lambda \tilde{\alpha}_{u}^{4}\int_{\mathbb{R}^{3}}\phi_{u^{+}}^{t}{u^{+}}^{2}dx+\lambda \tilde{\alpha}_{u}^{4}\int_{\mathbb{R}^{3}}\phi_{u^{-}}^{t}{u^{+}}^{2}dx
\leq\int_{\mathbb{R}^{3}}f(x,\tilde{\alpha}_{u} u^{+})\tilde{\alpha}_{u} u^{+}dx.
\end{align*}
Moreover, we have
\begin{align}
&\tilde{\alpha}_{u}^{-2}\Big(\int_{\mathbb{R}^{3}}\vert(-\Delta)^{\frac{s}{2}}u^{+}\vert^{2} dx + \int_{\mathbb{R}^{3}}(-\Delta)^{\frac{s}{2}}u^{+}(-\Delta)^{\frac{s}{2}}u^{-}dx+\int_{\mathbb{R}^{3}}V(x){u^{+}}^{2}dx\Big)\nonumber\\
&\quad\quad+\lambda \int_{\mathbb{R}^{3}}\phi_{u^{+}}^{t}{u^{+}}^{2}dx+\lambda \int_{\mathbb{R}^{3}}\phi_{u^{-}}^{t}{u^{+}}^{2}dx
\leq\int_{\mathbb{R}^{3}}\frac{f(x,\tilde{\alpha}_{u} u^{+})}{\tilde{\alpha}_{u}^{3}} u^{+}dx.
\end{align}
By (2.8) and (2.4), one has
\begin{align}
&(\tilde{\alpha}_{u}^{-2}-1)\Big(\int_{\mathbb{R}^{3}}\vert(-\Delta)^{\frac{s}{2}}u^{+}\vert^{2} dx + \int_{\mathbb{R}^{3}}(-\Delta)^{\frac{s}{2}}u^{+}(-\Delta)^{\frac{s}{2}}u^{-}dx+\int_{\mathbb{R}^{3}}V(x){u^{+}}^{2}dx\Big)\nonumber\\
&\quad\quad\quad\quad\quad\quad \quad\quad\quad \quad\quad\quad \quad\quad\quad  \leq\int_{\mathbb{R}^{3}}\Big(\frac{f(x,\tilde{\alpha}_{u} u^{+})}{(\tilde{\alpha}_{u}u^{+})^{3}}- \frac{f(x, u^{+})}{(u^{+})^{3}}\Big)(u^{+})^{4}dx.
\end{align}
By $(f_{4})$ and (2.8), it implies that $1\leq \tilde{\alpha}_{u}\leq \tilde{\beta}_{u}$. By the same method,  we may get $\tilde{\beta}_{u}\leq 1$ by $(f_{4})$, (2.5) and (2.7), it shows
that $\tilde{\alpha}_{u}=\tilde{\beta}_{u}=1$.\\

\textbf{Case 2.} $u\not\in \mathcal{M}_{\lambda}$.\\
If $u\not\in \mathcal{M}_{\lambda}$, then there exists a pair of positive numbers $(\alpha_{u}, \beta_{u})$ such that
$\alpha_{u}u^{+}+\beta_{u}u^{-}\in\mathcal{M}_{\lambda}$. Suppose that there exists another pair of positive numbers $(\alpha'_{u}, \beta'_{u})$
such that $\alpha'_{u}u^{+}+\beta'_{u}u^{-}\in\mathcal{M}_{\lambda}$. Set $v:=\alpha_{u}u^{+}+\beta_{u}u^{-}$ and $v':=\alpha'_{u}u^{+}+\beta'_{u}u^{-}$, we have
\begin{equation*}
\frac{\alpha'_{u}}{\alpha_{u}}v^{+}+\frac{\beta'_{u}}{\beta_{u}}v^{-}=\alpha'_{u}u^{+}+\beta'_{u}u^{-}=v'\in \mathcal{M}_{\lambda}.
\end{equation*}
Since $v\in \mathcal{M}_{\lambda}$, we obtain that $\alpha_{u}=\alpha'_{u}$ and $\beta_{u}=\beta'_{u}$, which implies that $(\alpha_{u}, \beta_{u})$
is the unique pair of numbers such that $\alpha_{u}u^{+}+\beta_{u}u^{-}\in\mathcal{M}_{\lambda}$. The proof is completed.
\end{proof}

\begin{lemma}
Assume that $(f_{1})-(f_{4})$ and $(V)$ hold. For a fixed $u\in H$ with $u^{\pm}\neq 0$. If $\langle I_{\lambda}'(u), u^{+}\rangle\leq 0$ and $\langle I_{\lambda}'(u), u^{-}\rangle\leq 0$, then there exists a unique pair $(\alpha_{u}, \beta_{u})\in (0, 1]\times (0, 1]$ such that $\langle I_{\lambda}'(\alpha_{u}u^{+}+ \beta_{u}u^{-} ), \alpha_{u}u^{+}\rangle=\langle I_{\lambda}'(\alpha_{u}u^{+}+ \beta_{u}u^{-}), \beta_{u}u^{-}\rangle =0$.
\end{lemma}

\begin{proof}
For $u\in H$ with $u^{\pm}\neq 0$, by Lemma 2.2, we know that there exist $\alpha_{u}$ and $\beta_{u}$ such that $\alpha_{u}u^{+}+\beta_{u}u^{-}\in \mathcal{M}_{\lambda}$. Without loss of generality, suppose that $\alpha_{u}\geq \beta_{u}>0$. Moreover, we have
\begin{align}
&\alpha_{u}^{2}\Big(\int_{\mathbb{R}^{3}}\vert(-\Delta)^{\frac{s}{2}}u^{+}\vert^{2} dx + \int_{\mathbb{R}^{3}}(-\Delta)^{\frac{s}{2}}u^{+}(-\Delta)^{\frac{s}{2}}u^{-}dx+\int_{\mathbb{R}^{3}}V(x){u^{+}}^{2}dx\Big)\nonumber\\
&\quad+\lambda \alpha_{u}^{4}\Big(\int_{\mathbb{R}^{3}}\phi_{u^{+}}^{t}{u^{+}}^{2}dx+\int_{\mathbb{R}^{3}}\phi_{u^{-}}^{t}{u^{+}}^{2}dx\Big)\nonumber\\
&\geq\alpha_{u}^{2}\int_{\mathbb{R}^{3}}\vert(-\Delta)^{\frac{s}{2}}u^{+}\vert^{2} dx +\alpha_{u}\beta_{u} \int_{\mathbb{R}^{3}}(-\Delta)^{\frac{s}{2}}u^{+}(-\Delta)^{\frac{s}{2}}u^{-}dx+\alpha_{u}^{2}\int_{\mathbb{R}^{3}}V(x){u^{+}}^{2}dx\nonumber\\
&\quad +\lambda \alpha_{u}^{4}\int_{\mathbb{R}^{3}}\phi_{u^{+}}^{t}{u^{+}}^{2}dx+\lambda \alpha_{u}^{2}\beta_{u}^{2}\int_{\mathbb{R}^{3}}\phi_{u^{-}}^{t}{u^{+}}^{2}dx\nonumber\\
&=\int_{\mathbb{R}^{3}}f(x,\alpha_{u} u^{+})\alpha_{u} u^{+}dx.
\end{align}
Since  $\langle I_{\lambda}'(u), u^{+}\rangle\leq 0$, it yields that
\begin{align}
&\int_{\mathbb{R}^{3}}\vert(-\Delta)^{\frac{s}{2}}u^{+}\vert^{2} dx + \int_{\mathbb{R}^{3}}(-\Delta)^{\frac{s}{2}}u^{+}(-\Delta)^{\frac{s}{2}}u^{-}dx+\int_{\mathbb{R}^{3}}V(x){u^{+}}^{2}dx\nonumber\\
&\quad \quad\quad+\lambda \int_{\mathbb{R}^{3}}\phi_{u^{+}}^{t}{u^{+}}^{2}dx+\lambda \int_{\mathbb{R}^{3}}\phi_{u^{-}}^{t}{u^{+}}^{2}dx\leq \int_{\mathbb{R}^{3}}f(x, u^{+})u^{+}dx.
\end{align}
Combine (2.10) and (2.11), we have
\begin{align*}
&\left(\frac{1}{\alpha_{u}^{2}}-1\right)\Big(\int_{\mathbb{R}^{3}}\vert(-\Delta)^{\frac{s}{2}}u^{+}\vert^{2} dx + \int_{\mathbb{R}^{3}}(-\Delta)^{\frac{s}{2}}u^{+}(-\Delta)^{\frac{s}{2}}u^{-}dx+\int_{\mathbb{R}^{3}}V(x){u^{+}}^{2}dx\Big)\\
&\quad\quad\quad\quad\quad\quad\geq \int_{\mathbb{R}^{3}}\Big(\frac{f(\alpha_{u}u^{+})}{(\alpha_{u}u^{+})^{3}}-\frac{f(u^{+})}{(u^{+})^{3}}\Big)(u^{+})^{4}dx.
\end{align*}
If $\alpha_{u}>1$, the left-hand side of this inequality  is negative. But from $(f_{4})$, the right-hand side of this inequality is positive, so have $\alpha_{u}\leq 1$.
The proof is thus complete.
\end{proof}

\begin{lemma}
For a fixed $u\in H$ with $u^{\pm}\neq 0$, then $(\alpha_{u}, \beta_{u})$ which obtained in Lemma 2.1 is the unique maximum point of the
function $\phi:\mathbb{R}_{+}\times \mathbb{R}_{+}\rightarrow \mathbb{R}$ defined as $\phi(s, t)=I_{\lambda}(\alpha u^{+}+\beta u^{-})$.
\end{lemma}

\begin{proof}
From the proof of Lemma 2.1, we know that $(\alpha_{u}, \beta_{u})$ is the unique critical point of $\phi$ in $\mathbb{R}_{+} \times \mathbb{R}_{+}$.
By $(f_{3})$, we conclude that $\phi(s, t)\rightarrow -\infty$ uniformly as $\vert (s, t)\vert\rightarrow\infty$, so it is
sufficient to show that a maximum point cannot be achieved on the boundary of $(\mathbb{R}_{+}, \mathbb{R}_{+})$. If we assume that
$(0, \bar{\beta})$ is a maximum point of $\phi$ with $\bar{\beta}\geq 0$. Then since
\begin{align*}
\phi(\alpha, \bar{\beta})&=I_{\lambda}(\alpha u^{+}+\bar{\beta}u^{-})\\
&=\frac{1}{2}\int_{\mathbb{R}^{3}}\Big(\vert(-\Delta)^{\frac{s}{2}}(\alpha u^{+}+\beta u^{-})\vert^{2}+V(x)(\alpha u^{+}+\beta u^{-})^{2}\Big)dx\\
&+\lambda\int_{\mathbb{R}^{3}}\phi_{\alpha u^{+}+\beta u^{-}}^{t}(\alpha u^{+}+\beta u^{-})^{2}dx-\int_{\mathbb{R}^{3}}f(x,\alpha u^{+}+\beta u^{-})(\alpha u^{+}+\beta u^{-})dx
\end{align*}
is an increasing function with respect to $\alpha$ if $\alpha$ is small enough, the pair $(0, \bar{\beta})$ is not a maximum point of $\phi$ in
$\mathbb{R}_{+}\times \mathbb{R}_{+}$. The proof is now finished.
\end{proof}

By Lemma 2.1, we define the minimization problem
\begin{equation*}
m_{\lambda}:=\inf\Big\{I_{\lambda}(u): u\in \mathcal{M}_{\lambda}\Big\}.
\end{equation*}

\begin{lemma}
Assume that $(f_{1})-(f_{4})$ and $(V_{1})-(V_{2})$ hold, then $m_{\lambda}>0$ can be achieved for any $\lambda>0$.
\end{lemma}

\begin{proof}
For every $u\in \mathcal{M}_{\lambda}$, we have $\langle I'_{\lambda}(u), u\rangle=0$. From $(f_{1})$, $(f_{2})$, for any
$\epsilon>0$, there exists $C_{\epsilon}>0$ such that
\begin{equation}
\vert f(x, u)u\vert\leq \epsilon u^{2}+C_{\epsilon}\vert u\vert^{p+1} \quad \text {for all}\, u\in \mathbb{R}.
\end{equation}
By Sobolev embedding theorem, we get
\begin{align}
\Vert u\Vert^{2}&\leq \int_{\mathbb{R}^{3}}\Big(\vert(-\Delta)^{\frac{s}{2}}u\vert^{2}+V(x)u^{2}\Big)dx+\lambda\int_{\mathbb{R}^{3}}\phi_{u}^{t}u^{2}dx=\int_{\mathbb{R}^{3}}f(x, u)udx\nonumber\\
&\leq \epsilon \int_{\mathbb{R}^{3}}\vert u\vert^{2}dx+C_{\epsilon}\int_{\mathbb{R}^{3}}\vert u\vert^{p+1}dx\nonumber\\
&\leq C_{2}\epsilon \Vert u\Vert^{2}+C_{\epsilon}'\Vert u\Vert^{p+1}
\end{align}
Pick $\epsilon=\frac{1}{2C_{2}}$. So there exists a constant $\gamma>0$ such that $\Vert u\Vert^{2}>\gamma$.

By $(f_{4})$, we have
\begin{equation*}
f(x, u)u-4F(x, u)\geq 0.
\end{equation*}
Then
\begin{equation}
 I_{\lambda}(u)= I_{\lambda}(u)-\frac{1}{4}\langle I'_{\lambda}(u), u\rangle\geq \frac{\Vert u\Vert^{2}}{4}\geq \frac{\gamma}{4}.
\end{equation}
This implies that $I_{\lambda}(u)$ is coercive in $\mathcal{M}_{\lambda}$ and $m_{\lambda}\geq \frac{\gamma}{4}>0$.

Let $\{u_{n}\}_n\subset \mathcal{M}_{\lambda}$ be such that $ I_{\lambda}(u_{n})\rightarrow m_{\lambda}$. Then $\{u_{n}\}_n$ is bounded in $H$ by (2.14).
Using Lemma 1.1, up to a subsequence, we have
\begin{align}
&u_{n}^{\pm}\rightharpoonup u_{\lambda}^{\pm} \quad \text{weakly in}\, H,\nonumber\\
&u_{n}^{\pm}\rightarrow u_{\lambda}^{\pm} \quad \text{strongly in}\, L^{q}(\mathbb{R}^{3}),\quad \text{for}\,\, q\in [2, 2_{s}^{*}),\\
&u_{n}^{\pm}\rightarrow u_{\lambda}^{\pm} \quad \text{a.e.}\,\, \text{in}\, \, \mathbb{R}^{3}\nonumber\\,
&(-\Delta)^{\frac{s}{2}}u_{n}^{\pm}\rightarrow (-\Delta)^{\frac{s}{2}}u_{\lambda}^{\pm}\quad \text{a.e. in}\, \mathbb{R}^{3}.\nonumber
\end{align}
Moreover, the conditions $(f_{1})$, $(f_{2})$ and Lemma 1.1 imply that
\begin{equation}
 \underset{n\rightarrow\infty}{\lim} \int_{\mathbb{R}^{3}}f(x, u_{n}^{\pm})u_{n}^{\pm}dx=\int_{\mathbb{R}^{3}}f(x, u_{\lambda}^{\pm})u_{\lambda}^{\pm}dx,\quad\,\,
 \underset{n\rightarrow\infty}{\lim} \int_{\mathbb{R}^{3}}F(x, u_{n}^{\pm})dx=\int_{\mathbb{R}^{3}}F(x, u_{\lambda}^{\pm})dx.
\end{equation}
 Since $u_{n}\in \mathcal{M}_{\lambda}$, we have
$\langle I'_{\lambda}(u_{n}), u_{n}^{\pm}\rangle=0$, that is
\begin{align}
&\int_{\mathbb{R}^{3}}\vert(-\Delta)^{\frac{s}{2}}u_{n}^{+}\vert^{2} dx + \int_{\mathbb{R}^{3}}(-\Delta)^{\frac{s}{2}}u_{n}^{+}(-\Delta)^{\frac{s}{2}}u_{n}^{-}dx
+\int_{\mathbb{R}^{3}}V(x){u_{n}^{+}}^{2}dx\nonumber\\
&\quad\quad+\lambda \int_{\mathbb{R}^{3}}\phi_{u_{n}^{+}}^{t}{u_{n}^{+}}^{2}dx+\lambda \int_{\mathbb{R}^{3}}\phi_{u_{n}^{-}}^{t}{u_{n}^{+}}^{2}dx
=\int_{\mathbb{R}^{3}}f(x, u_{n}^{+}) u_{n}^{+}dx,
\end{align}
and
\begin{align}
&\int_{\mathbb{R}^{3}}\vert(-\Delta)^{\frac{s}{2}}u_{n}^{-}\vert^{2} dx +\int_{\mathbb{R}^{3}}(-\Delta)^{\frac{s}{2}}u_{n}^{+}(-\Delta)^{\frac{s}{2}}u_{n}^{-}dx
+\int_{\mathbb{R}^{3}}V(x){u_{n}^{-}}^{2}dx\nonumber\\
&\quad\quad+\lambda \int_{\mathbb{R}^{3}}\phi_{u_{n}^{-}}^{t}{u_{n}^{-}}^{2}dx+\lambda \int_{\mathbb{R}^{3}}\phi_{u_{n}^{+}}^{t}{u_{n}^{-}}^{2}dx
=\int_{\mathbb{R}^{3}}f(x, u_{n}^{-}) u_{n}^{-}dx.
\end{align}

Similar as (2.12) and (2.13),  we also have $\Vert u_{n}^{\pm}\Vert^{2}\geq \delta$ for all $n\in N$, where $\delta$ is a constant.

Since $u_{n}\in \mathcal{M}_{\lambda}$, by (2.17) and (2.18) again, we have
\begin{align*}
\delta\leq\Vert u_{n}^{\pm}\Vert^{2}<
\int_{\mathbb{R}^{3}}f(x, u_{n}^{\pm})u_{n}^{\pm} dx&\leq \epsilon\int_{\mathbb{R}^{3}}\vert u_{n}^{\pm}\vert^{2}dx+
C_{\epsilon}\int_{\mathbb{R}^{3}}\vert u_{n}^{\pm}\vert^{p+1}dx\\
&\leq \frac{\epsilon}{V_{0}}\int_{\mathbb{R}^{3}}\vert u_{n}^{\pm}\vert^{2}dx+
C_{\epsilon}\int_{\mathbb{R}^{3}}\vert u_{n}^{\pm}\vert^{p+1}dx.
\end{align*}
Using the boundedness of $\{u_{n}\}_n$, there is $C_{2}>0$ such that
\begin{equation*}
\delta\leq\epsilon C_{2}+C_{\epsilon}\int_{\mathbb{R}^{3}}\vert u_{n}^{\pm}\vert^{p+1}dx.
\end{equation*}
Choosing $\epsilon=\delta/(2C_{2})$, we get
\begin{equation}
\int_{\mathbb{R}^{3}}\vert u_{n}^{\pm}\vert^{p+1}dx\geq \frac{\delta}{2\bar{C}}.
\end{equation}
where $\bar{C}$ is a positive constant, in fact, $\bar{C}=C_{\frac{\delta}{2C_{2}}}$.

By (2.19) and Lemma 1.1 (ii),  we get
\begin{equation*}
\int_{\mathbb{R}^{3}}\vert u_{\lambda}^{\pm}\vert^{p+1}dx\geq \frac{\delta}{2\bar{C}}.
\end{equation*}
Thus, $u_{\lambda}^{\pm}\neq 0$. From Lemma 2.1, there exists $\alpha_{u_{\lambda}}$, $\beta_{u_{\lambda}}>0$ such that
\begin{equation*}
\bar{u}_{\lambda}:=\alpha_{u_{\lambda}}u_{\lambda}^{+}+\beta_{u_{\lambda}}u_{\lambda}^{-}\in \mathcal{M}_{\lambda}.
\end{equation*}
Now, we show that  $\alpha_{u_{\lambda}}$, $\beta_{u_{\lambda}}\leq 1$. By (2.15), (2.17), the weak semicontinuity of norm, Fatou's Lemma and Lemma 1.2,
 we have
\begin{align}
&\int_{\mathbb{R}^{3}}\vert(-\Delta)^{\frac{s}{2}}u_{\lambda}^{+}\vert^{2} dx + \int_{\mathbb{R}^{3}}(-\Delta)^{\frac{s}{2}}u_{\lambda}^{+}(-\Delta)^{\frac{s}{2}}u_{\lambda}^{-}dx
+\int_{\mathbb{R}^{3}}V(x){u_{\lambda}^{+}}^{2}dx\nonumber\\
&\quad \quad+\lambda \int_{\mathbb{R}^{3}}\phi_{u_{\lambda}^{+}}^{t}{u_{\lambda}^{+}}^{2}dx+\lambda \int_{\mathbb{R}^{3}}\phi_{u_{\lambda}^{-}}^{t}{u_{\lambda}^{+}}^{2}dx
\leq\int_{\mathbb{R}^{3}}f(x, u_{\lambda}^{+}) u_{\lambda}^{+}dx.
\end{align}
From (2.20) and Lemma 2.2, we have $\alpha_{u_{\lambda}}\leq 1$. Similarly, $\beta_{u_{\lambda}}\leq 1$.
The condition $(f_{4})$ implies that $H(u):=uf(x, u)-4F(x, u)$ is a non-negative function, increasing in $\vert u\vert$, so we have
\begin{align*}
m_{\lambda}&\leq I_{\lambda}(\bar{u}_{\lambda})=I_{\lambda}(\bar{u}_{\lambda})-\frac{1}{4}\langle  I'_{\lambda}(\bar{u}_{\lambda}), \bar{u}_{\lambda}\rangle\\
 &=\frac{1}{4}\Vert \bar{u}_{\lambda}\Vert^{2}+\frac{1}{4}\int_{\mathbb{R}^{3}}\Big(f(\bar{u}_{\lambda})\bar{u}_{\lambda}-4F(\bar{u}_{\lambda})\Big)dx\\
&= \frac{1}{4}\Vert \alpha_{u_{\lambda}}u_{\lambda}^{+}\Vert^{2}+\frac{1}{4}\Vert \beta_{u_{\lambda}}u_{\lambda}^{-}\Vert^{2}+\frac{\alpha_{u_{\lambda}} \beta_{u_{\lambda}}}{2}\int_{\mathbb{R}^{3}}(-\Delta)^{\frac{s}{2}}u_{\lambda}^{+}(-\Delta)^{\frac{s}{2}}u_{\lambda}^{-}dx\\
&+\frac{1}{4}\int_{\mathbb{R}^{3}}\Big(f(\alpha_{u_{\lambda}}u_{\lambda}^{+})\alpha_{u_{\lambda}}u_{\lambda}^{+}
-4F(x, \alpha_{u_{\lambda}}u_{\lambda}^{+})\Big)dx+\frac{1}{4}\int_{\mathbb{R}^{3}}\Big(f(x, \beta_{u_{\lambda}}u_{\lambda}^{-})\beta_{u_{\lambda}}u_{\lambda}^{-}
-4F(x, \beta_{u_{\lambda}}u_{\lambda}^{-})\Big)dx\\
&\leq \frac{1}{4}\Vert u_{\lambda}^{+}\Vert^{2}+\frac{1}{4}\Vert u_{\lambda}^{-}\Vert^{2}+\frac{1}{2}\int_{\mathbb{R}^{3}}(-\Delta)^{\frac{s}{2}}u_{\lambda}^{+}(-\Delta)^{\frac{s}{2}}u_{\lambda}^{-}dx\\
&+\frac{1}{4}\int_{\mathbb{R}^{3}}\Big(f(x, u_{\lambda}^{+})u_{\lambda}^{+}
-4F(x, u_{\lambda}^{+})\Big)dx+\frac{1}{4}\int_{\mathbb{R}^{3}}\Big(f(x, u_{\lambda}^{-})u_{\lambda}^{-}
-4F(x, u_{\lambda}^{-})\Big)dx\\
&\leq \underset{n\rightarrow \infty}{\lim\inf}\Big[I_{\lambda}(u_{n})-\frac{1}{4}\langle  I'_{\lambda}(u_{n}), u_{n}\rangle\Big]=m_{\lambda}.
\end{align*}
We then conclude that $\alpha_{u_{\lambda}}=\beta_{u_{\lambda}}=1$. Thus, $\bar{u}_{\lambda}=u_{\lambda}$ and $I_{\lambda}(u_{\lambda})=m_{\lambda}$.\\
\end{proof}

\section{Proof of Main Results} \label{pf1}

In this section, we are devoted to proving our main results.

\begin{proof}[Proof of Theorem 1.1]
We firstly prove that the minimizer $u_{\lambda}$ for the minimization problem is indeed a sign-changing solution of problem (1.4), that is, $I'_{\lambda}(u_{\lambda})=0$. For it, we will use the quantitative deformation lemma.

It is clear that $I'_{\lambda}(u_{\lambda})u_{\lambda}^{+}=0=I'_{\lambda}(u_{\lambda})u_{\lambda}^{-}$. From Lemma 2.2, for any $(\alpha, \beta)\in \mathbb{R}_{+}\times \mathbb{R}_{+}$ and
$(\alpha, \beta)\neq (1, 1)$,
\begin{equation*}
I_{\lambda}(\alpha u_{\lambda}^{+}+\beta u_{\lambda}^{-})<I_{\lambda}( u_{\lambda}^{+}+u_{\lambda}^{-})=m_{\lambda}.
\end{equation*}
If $I'_{\lambda}(u_{\lambda})\neq 0$, then there exist $\delta>0$ and $\kappa>0$ such that
\begin{equation*}
\Vert I'_{\lambda}(v)\Vert\geq \kappa \quad \text{for all} \,\Vert v-u_{\lambda}\Vert\leq 3\delta.
\end{equation*}
Let $D:=(\frac{1}{2}, \frac{3}{2})\times (\frac{1}{2}, \frac{3}{2})$ and $g(\alpha, \beta):=\alpha u_{\lambda}^{+}+\beta u_{\lambda}^{-}$. From Lemma 2.3, we also have
\begin{equation*}
\bar{m}_{\lambda}:=\underset{\partial D}{\max} I_{\lambda}\circ g<m_{\lambda}.
\end{equation*}
For $\epsilon:=\min\{(m_{\lambda}-\bar{m}_{\lambda})/2, \kappa\delta/8\}$ and $S:=B(u_{\lambda}, \delta)$, there is a deformation $\eta$ such that
\begin{itemize}
\item[$(a)$] $\eta(1, u)=u$ if $u\not\in I_{\lambda}^{-1}([m_{\lambda}-2\epsilon, m_{\lambda}+2\epsilon])\cap S_{2\delta}$;
\item[$(b)$] $\eta(1, I_{\lambda}^{m_{\lambda}+\epsilon}\cap S)\subset I_{\lambda}^{m_{\lambda}-\epsilon}$;
\item[$(c)$] $I_{\lambda}(\eta(1, u)))\leq I_{\lambda}(u)$ for all $u\in H$.
\end{itemize}
See \cite{rW} for more details. It is clear that
\begin{equation*}
\underset{(\alpha, \beta)\in \bar{D}}{\max}I_{\lambda}(\eta(1, g(\alpha, \beta))))<m_{\lambda}.
\end{equation*}
 Now we prove that $\eta(1, g(D))\cap \mathcal{M}_{\lambda}\neq \emptyset$ which contradicts to the definition of $m_{\lambda}$. Let us define
 $h(\alpha, \beta)=\eta(1, g(\alpha, \beta)))$ and
\begin{equation*}
\Psi_{0}(\alpha, \beta):=\Big(I'_{\lambda}(g(\alpha, \beta))u_{\lambda}^{+}, I'_{\lambda}(g(\alpha, \beta))u_{\lambda}^{-}\Big)=\Big(I'_{\lambda}(\alpha u_{\lambda}^{+}+\beta u_{\lambda}^{-})u_{\lambda}^{+}, I'_{\lambda}(\alpha u_{\lambda}^{+}+\beta u_{\lambda}^{-})u_{\lambda}^{-}\Big),
\end{equation*}
\begin{equation*}
\Psi_{1}(\alpha, \beta):=\Big(\frac{1}{\alpha}I'_{\lambda}(h(\alpha, \beta))h^{+}(\alpha, \beta), \frac{1}{\beta}I'_{\lambda}(h(\alpha, \beta))h^{-}(\alpha, \beta)\Big).
\end{equation*}
Lemma 2.1 and the degree theory imply that $\deg (\Psi_{0}, D, 0)=1$. it follows from that $g=h$ on $\partial D$. Consequently,
we obtain
\begin{equation*}
\deg (\Psi_{1}, D, 0)=\deg (\Psi_{0}, D, 0)=1.
\end{equation*}
 Thus, $\Psi_{1}(\alpha_{0}, \beta_{0})=0$ for some $(\alpha_{0}, \beta_{0})\in D$, so that
 \begin{equation*}
\eta(1, g(\alpha_{0}, \beta_{0})))=h(\alpha_{0}, \beta_{0})\in \mathcal{M}_{\lambda},
\end{equation*}
which is a contradiction. From this, $u_{\lambda}$ is a critical point of
$I_{\lambda}$, moreover, it is a sign-changing solution for problem (1.4).

Now we prove that $u_{\lambda}$ has exactly two nodal domains. By contradiction, we assume that $u_{\lambda}$ has at least three nodal domains
$\Omega_{1}$, $\Omega_{2}$, $\Omega_{3}$. Without loss generality, we may assume that $u_{\lambda}\geq 0$ a.e. in $\Omega_{1}$ and $u_{\lambda}\leq0$ a.e. in $\Omega_{2}$.
Set
\begin{equation*}
{u_{\lambda_{i}}}:=\chi_{\Omega_{i}}u_{\lambda}, \quad\quad i=1, 2, 3,
\end{equation*}
where
\begin{equation*}
\chi_{\Omega_{i}}=\left\{
\begin{array}{l}
1\quad \quad \quad   \quad x\in \Omega_{i},\\
0\quad \quad \quad\quad x\in \mathbb{R}^{3}\setminus \Omega_{i}.\\
\end{array}%
\right.
\end{equation*}%
So $\text{suppt}(u_{\lambda_{1}})\cap \text{suppt}(u_{\lambda_{2}})=\emptyset$, ${u_{\lambda_{i}}}\neq 0$ and $\langle I'(u_{\lambda}), {u_{\lambda_{i}}}\rangle=0$ for $i=1, 2, 3$. Assume that $v:=u_{\lambda_{1}}+u_{\lambda_{2}}$, then $v^{+}=u_{\lambda_{1}}$ and $v^{-}=u_{\lambda_{2}}$, i.e. $v^{\pm}\neq 0$. By Lemma 2.1,
there is a unique pair $(\alpha_{v}, \beta_{v})$ of positive numbers such that
\begin{equation*}
\alpha_{v}v^{+}+\beta_{v}v^{+}\in \mathcal{M}_{\lambda},
\end{equation*}
 so we have
 \begin{equation*}
I(\alpha_{v}{u_{\lambda_{1}}}+\beta_{v}{u_{\lambda_{2}}})\geq m_{\lambda}.
\end{equation*}
From $\langle I'(u_{\lambda}), {u_{\lambda_{i}}}\rangle=0$ for $i=1, 2, 3$, we have
 \begin{equation*}
\langle I'(v), v^{\pm}\rangle<0.
\end{equation*}
By Lemma 2.2, we know that $(\alpha_{v}, \beta_{v})\in (0, 1]\times (0, 1]$.
Since
\begin{align*}
0=\frac{1}{4}\langle  I'_{\lambda}(u_{\lambda}), {u_{\lambda_{3}}}\rangle&=\frac{1}{4}\Vert {u_{\lambda_{3}}}\Vert^{2}+\frac{1}{4} \int_{\mathbb{R}^{3}}(-\Delta)^{\frac{s}{2}}{u_{\lambda_{1}}}(-\Delta)^{\frac{s}{2}}{u_{\lambda_{3}}}dx+
\frac{1}{4} \int_{\mathbb{R}^{3}}(-\Delta)^{\frac{s}{2}}{u_{\lambda_{2}}}(-\Delta)^{\frac{s}{2}}{u_{\lambda_{3}}}dx \\
 &\quad +\frac{\lambda}{4}\int_{\mathbb{R}^{3}}\phi^{t}_{u_{\lambda_{1}}}{u_{\lambda_{3}}}^{2}dx
 +\frac{\lambda}{4}\int_{\mathbb{R}^{3}}\phi^{t}_{u_{\lambda_{2}}}{u_{\lambda_{3}}}^{2}dx
 +\frac{\lambda}{4}\int_{\mathbb{R}^{3}}\phi^{t}_{u_{\lambda_{3}}}{u_{\lambda_{3}}}^{2}dx\\
&\quad -\frac{1}{4}\int_{\mathbb{R}^{3}}f(x, {u_{\lambda_{3}}}){u_{\lambda_{3}}}dx \\
&\leq\frac{1}{4}\Vert {u_{\lambda_{3}}}\Vert^{2}+\frac{1}{4} \int_{\mathbb{R}^{3}}(-\Delta)^{\frac{s}{2}}{u_{\lambda_{1}}}(-\Delta)^{\frac{s}{2}}{u_{\lambda_{3}}}dx+
\frac{1}{4} \int_{\mathbb{R}^{3}}(-\Delta)^{\frac{s}{2}}{u_{\lambda_{2}}}(-\Delta)^{\frac{s}{2}}{u_{\lambda_{3}}}dx \\
 &\quad +\frac{\lambda}{4}\int_{\mathbb{R}^{3}}\phi^{t}_{u_{\lambda_{1}}}{u_{\lambda_{3}}}^{2}dx
 +\frac{\lambda}{4}\int_{\mathbb{R}^{3}}\phi^{t}_{u_{\lambda_{2}}}{u_{\lambda_{3}}}^{2}dx
 +\frac{\lambda}{4}\int_{\mathbb{R}^{3}}\phi^{t}_{u_{\lambda_{3}}}{u_{\lambda_{3}}}^{2}dx-\int_{\mathbb{R}^{3}}F(x, {u_{\lambda_{3}}})dx\\
&< I_{\lambda}({u_{\lambda_{3}}})++\frac{1}{4} \int_{\mathbb{R}^{3}}(-\Delta)^{\frac{s}{2}}{u_{\lambda_{1}}}(-\Delta)^{\frac{s}{2}}{u_{\lambda_{3}}}dx+
\frac{1}{4} \int_{\mathbb{R}^{3}}(-\Delta)^{\frac{s}{2}}{u_{\lambda_{2}}}(-\Delta)^{\frac{s}{2}}{u_{\lambda_{3}}}dx\\
&\quad +\frac{\lambda}{4}\int_{\mathbb{R}^{3}}\phi^{t}_{u_{\lambda_{1}}}{u_{\lambda_{3}}}^{2}dx
 +\frac{\lambda}{4}\int_{\mathbb{R}^{3}}\phi^{t}_{u_{\lambda_{2}}}{u_{\lambda_{3}}}^{2}dx.
\end{align*}
From $(f_{4})$, we have
\begin{align*}
m_{\lambda}&\leq I_{\lambda}(\alpha_{v}{u_{\lambda_{1}}}+\beta_{v}{u_{\lambda_{2}}})\\
&= I_{\lambda}(\alpha_{v}{u_{\lambda_{1}}}+\beta_{v}{u_{\lambda_{2}}})-\frac{1}{4}\langle I'_{\lambda}(\alpha_{v}{u_{\lambda_{1}}}+\beta_{v}{u_{\lambda_{2}}}), \alpha_{v}{u_{\lambda_{1}}}+\beta_{v}{u_{\lambda_{2}}}\rangle\\
&=\frac{\Vert \alpha_{v}{u_{\lambda_{1}}}+\beta_{v} {u_{\lambda_{2}}}\Vert^{2}}{4}+\int_{\mathbb{R}^{3}}\Big(\frac{1}{4}f(x, \alpha_{v}{u_{\lambda_{1}}})\alpha_{v}{u_{\lambda_{1}}}-F(x, \alpha_{v}{u_{\lambda_{1}}})\Big)dx\\
& \quad +\int_{\mathbb{R}^{3}}\Big(\frac{1}{4}f(x, \beta_{v}{u_{\lambda_{2}}})\beta_{v}{u_{\lambda_{2}}}-F(x, \beta_{v}{u_{\lambda_{2}}})\Big)dx\\
&\leq \frac{\Vert {u_{\lambda_{1}}}\Vert^{2}+\Vert {u_{\lambda_{2}}}\Vert^{2}}{4}+\frac{1}{2}\int_{\mathbb{R}^{3}}(-\Delta)^{\frac{s}{2}}{u_{\lambda_{1}}}(-\Delta)^{\frac{s}{2}}{u_{\lambda_{1}}}dx\\
&\quad +\int_{\mathbb{R}^{3}}\Big(\frac{1}{4}f(x,{u_{\lambda_{1}}}){u_{\lambda_{1}}}-F(x,{u_{\lambda_{1}}})\Big)dx+\int_{\mathbb{R}^{3}}\Big(\frac{1}{4}f(x, {u_{\lambda_{2}}}){u_{\lambda_{2}}}-F(x, {u_{\lambda_{2}}})\Big)dx\\
&=I_{\lambda}({u_{\lambda_{1}}})+I_{\lambda}({u_{\lambda_{2}}})
+\int_{\mathbb{R}^{3}}(-\Delta)^{\frac{s}{2}}{u_{\lambda_{1}}}(-\Delta)^{\frac{s}{2}}{u_{\lambda_{2}}}dx
+\frac{1}{4}\int_{\mathbb{R}^{3}}(-\Delta)^{\frac{s}{2}}{u_{\lambda_{3}}}(-\Delta)^{\frac{s}{2}}{u_{\lambda_{1}}}dx\\
&\quad +\frac{1}{4}\int_{\mathbb{R}^{3}}(-\Delta)^{\frac{s}{2}}{u_{\lambda_{3}}}(-\Delta)^{\frac{s}{2}}{u_{\lambda_{2}}}dx+
\frac{\lambda}{4}\int_{\mathbb{R}^{3}}\phi^{t}_{{u_{\lambda_{2}}}}{u_{\lambda_{1}}}^{2}dx
+\frac{\lambda}{4}\int_{\mathbb{R}^{3}}\phi^{t}_{{u_{\lambda_{3}}}}{u_{\lambda_{1}}}^{2}dx\\
&\quad+\frac{\lambda}{4}\int_{\mathbb{R}^{3}}\phi^{t}_{{u_{\lambda_{1}}}}{u_{\lambda_{2}}}^{2}dx
+\frac{\lambda}{4}\int_{\mathbb{R}^{3}}\phi^{t}_{{u_{\lambda_{3}}}}{u_{\lambda_{2}}}^{2}dx\\
&<I_{\lambda}({u_{\lambda_{1}}})+I_{\lambda}({u_{\lambda_{2}}})+I_{\lambda}({u_{\lambda_{3}}})+
\int_{\mathbb{R}^{3}}(-\Delta)^{\frac{s}{2}}{u_{\lambda_{1}}}(-\Delta)^{\frac{s}{2}}{u_{\lambda_{2}}}dx\\
&\quad +\int_{\mathbb{R}^{3}}(-\Delta)^{\frac{s}{2}}{u_{\lambda_{1}}}(-\Delta)^{\frac{s}{2}}{u_{\lambda_{3}}}dx+
\int_{\mathbb{R}^{3}}(-\Delta)^{\frac{s}{2}}{u_{\lambda_{2}}}(-\Delta)^{\frac{s}{2}}{u_{\lambda_{3}}}dx
+\frac{\lambda}{4}\int_{\mathbb{R}^{3}}\phi^{t}_{{u_{\lambda_{2}}}}{u_{\lambda_{1}}}^{2}dx
\end{align*}
\begin{align*}
&+\frac{\lambda}{4}\int_{\mathbb{R}^{3}}\phi^{t}_{{u_{\lambda_{3}}}}{u_{\lambda_{1}}}^{2}dx
+\frac{\lambda}{4}\int_{\mathbb{R}^{3}}\phi^{t}_{{u_{\lambda_{1}}}}{u_{\lambda_{2}}}^{2}dx
+\frac{\lambda}{4}\int_{\mathbb{R}^{3}}\phi^{t}_{{u_{\lambda_{3}}}}{u_{\lambda_{2}}}^{2}dx\nonumber\\
&+\frac{\lambda}{4}\int_{\mathbb{R}^{3}}\phi^{t}_{{u_{\lambda_{1}}}}{u_{\lambda_{3}}}^{2}dx+
\frac{\lambda}{4}\int_{\mathbb{R}^{3}}\phi^{t}_{{u_{\lambda_{2}}}}{u_{\lambda_{3}}}^{2}dx\nonumber \\
&=I_{\lambda}(u_{\lambda})=m_{\lambda},
\end{align*}
which is impossible, so $u_{\lambda}$ has exactly two nodal domains.
\end{proof}

\begin{proof}[Proof of Theorem 1.2]
Similar as the proof of Lemma 2.4, for each $\lambda>0$, we can get a $v_{\lambda}\in \mathcal{N}_{\lambda}$ such that
$I_{\lambda}(v_{\lambda})=c_{\lambda}>0$, where $\mathcal{N}_{\lambda}$ and $c_{\lambda}$ are defined by (1.5) and (1.6), respectively.
Moreover, the critical points of $I_{\lambda}$ on $\mathcal{N}_{\lambda}$ are the critical points of $I_{\lambda}$ in $H$. Thus, $v_{\lambda}$
is a ground state solution of problem (1.4).

From Theorem 1.1, we know that problem (1.4) has a least energy sign-changing solution $u_{\lambda}$ which changes sign only once. Suppose that
$u_{\lambda}=u_{\lambda}^{+}+u_{\lambda}^{-}$.  As the proof of Step 1 in Lemma 2.1, there is a unique $\alpha_{u_{\lambda}^{+}}>0$ such that
\begin{equation*}
\alpha_{u_{\lambda}^{+}}u_{\lambda}^{+}\in \mathcal{N}_{\lambda}.
\end{equation*}
Similarly, there exists a unique $\beta_{u_{\lambda}^{-}}>0$, such that
\begin{equation*}
\beta_{u_{\lambda}^{-}}u_{\lambda}^{-}\in \mathcal{N}_{\lambda}.
\end{equation*}
 Moreover, Lemma 2.2 implies that $\alpha_{u_{\lambda}^{+}}, \beta_{u_{\lambda}^{-}}\in (0, 1]$.
Therefore, by Lemma 2.3, we obtain that
\begin{equation*}
2c_{\lambda}\leq I_{\lambda}(\alpha_{u_{\lambda}^{+}}u_{\lambda}^{+})+I_{\lambda}(\beta_{u_{\lambda}^{-}}u_{\lambda}^{-}) \leq I_{\lambda}(\alpha_{u_{\lambda}^{+}}u_{\lambda}^{+}+\beta_{u_{\lambda}^{-}}u_{\lambda}^{-})\leq I_{\lambda}(u_{\lambda}^{+}+u_{\lambda}^{-})=m_{\lambda}
\end{equation*}
that is $I_{\lambda}(u_{\lambda}) \geq 2c_{\lambda}$. It follows that  $c_{\lambda}>0$ which cannot be achieved by a sign-changing function. This completes the proof.
\end{proof}

Now we prove Theorem 1.3.  In the following, we regard $\lambda>0$ as a parameter in problem (1.1). We shall study the convergence property of $u_{\lambda}$
as $\lambda\searrow 0$.

\begin{proof}[Proof of Theorem 1.3]

For any $\lambda>0$, let $u_{\lambda}\in H$ be the least energy sign-changing solution of problem (1.1) obtained in Theorem 1.1, which has exactly two nodal domains.

\vspace{2mm}

\textbf{Step 1.} We show that, for any sequence $\{\lambda_{n}\}_n$ with $\lambda_{n} \searrow 0$ as $n\rightarrow \infty$, $\{{u_{\lambda_{n}}}\}_n$
is bounded in $H$. \\

Choose a nonzero function $\varphi \in C_{0}^{\infty}(\mathbb{R}^{3})$ with $\varphi^{\pm}\neq 0$. By $(f_{3})$ and $(f_{4})$,  for any $\lambda\in[0, 1]$, there exists a pair $(\theta_{1}, \theta_{2})\in(\mathbb{R}_{+}\times \mathbb{R}_{+})$, which does not depend on $\lambda$, such that
\begin{equation*}
\langle I'_{\lambda}(\theta_{1}\varphi^{+}+\theta_{2}\varphi^{-}), \theta_{1}\varphi^{+}\rangle<0 \quad\text{that}\quad \langle I'_{\lambda}(\theta_{1}\varphi^{+}+\theta_{2}\varphi^{-}), \theta_{2}\varphi^{-}\rangle<0.
\end{equation*}
Then in view of Lemmas 2.1 and Lemma 2.2, for any $\lambda\in[0, 1]$, there is a unique pair $(\alpha_{\varphi}(\lambda), \beta_{\varphi}(\lambda) )\in(0, 1]\times(0, 1]$ such that $\bar{\varphi}:=\alpha_{\varphi}(\lambda)\theta_{1}\varphi^{+}+\beta_{\varphi}(\lambda)\theta_{2}\varphi^{-}\in \mathcal{M}_{\lambda}$. Thus, for all $\lambda\in[0, 1]$, we have
\begin{align*}
 I_{\lambda}(u_{\lambda})&\leq  I_{\lambda}(\bar{\varphi})
= I_{\lambda}(\bar{\varphi})-\frac{1}{4}\langle I'_{\lambda}(\bar{\varphi}), \bar{\varphi}\rangle\\
&=\frac{\Vert \bar{\varphi}\Vert^{2}}{4}+
\int_{\mathbb{R}^{3}}\Big(\frac{1}{4}f(x, \bar{\varphi})\bar{\varphi}- F(x, \bar{\varphi}) \Big)dx\\
&\leq \frac{\Vert \bar{\varphi}\Vert^{2}}{4}+
\int_{\mathbb{R}^{3}}\Big(C_{3}\vert \bar{\varphi}\vert^{2}+ C_{4}\vert \bar{\varphi}\vert^{p+1}\Big)dx\\
&\leq \frac{\Vert \theta_{1}\varphi^{+}\Vert^{2}}{4}+\frac{\Vert \theta_{2}\varphi^{-}\Vert^{2}}{4}+ \frac{1}{2} \int_{\mathbb{R}^{3}}(-\Delta)^{\frac{s}{2}}({\theta_{1}\varphi^{+}})(-\Delta)^{\frac{s}{2}}({\theta_{2}\varphi^{-}})dx\\
&\quad+\int_{\mathbb{R}^{3}}\Big(C_{3}\vert\theta_{1}\varphi^{+}\vert^{2}+ C_{4}\vert \theta_{1}\varphi^{+}\vert^{p+1}+C_{3}\vert\theta_{2}\varphi^{-}\vert^{2}+ C_{4}\vert \theta_{2}\varphi^{-}\vert^{p+1}\Big)dx\\
&=C_{0}.
\end{align*}
Moreover, for $n$ large enough, we obtain
\begin{equation*}
C_{0}+1\geq I_{\lambda_{n}}(u_{\lambda_{n}})=I_{\lambda_{n}}(u_{\lambda_{n}})-\frac{1}{4}\langle I'_{\lambda_{n}}(u_{\lambda_{n}}), u_{\lambda_{n}}\rangle\geq \frac{1}{4}\Vert u_{\lambda_{n}}\Vert^{2}.
\end{equation*}
So $\{{u_{\lambda}}_{n}\}_n$ is bounded in $H$.

\vspace{2mm}

\textbf{Step 2.} The problem has a sign-changing solution $u_{0}$.\\

By step 1 and Lemma 1.1, there exists a subsequence of $\{\lambda_{n}\}_n$, still denoted by $\{\lambda_{n}\}_n$ and $u_{0}\in H$ such that
\begin{align}
&u_{\lambda_{n}}\rightharpoonup u_{0}\quad \text{weakly in}\quad H,\nonumber\\
&u_{\lambda_{n}}\rightarrow u_{0}\quad \text{strongly in}\, L^{q}(\mathbb{R}^{3})\quad \text{for}\, q\in[2, 2_{s}^{*}),\\
&u_{\lambda_{n}}\rightarrow u_{0}\quad  \text{a.e. in}\,\, \mathbb{R}^{3}\nonumber.
\end{align}
Since $u_{\lambda_{n}}$ is the least energy sign-changing solution of (1.4) with $\lambda=\lambda_{n}$, then we have\\
\begin{align*}
 \int_{\mathbb{R}^{3}}\Big((-\Delta)^{\frac{s}{2}}u_{\lambda_{n}}(-\Delta)^{\frac{s}{2}}v+V(x)u_{\lambda_{n}}v\Big)dx+\lambda_{n}\int_{\mathbb{R}^{3}}\phi_{u_{\lambda_{n}}}^{t}u_{\lambda_{n}}v dx=\int_{\mathbb{R}^{3}}f(x,u_{\lambda_{n}})v dx.
\end{align*}
 for all $v\in C_{0}^{\infty}(\mathbb{R}^{3})$. From (3.1), we get that
 \begin{align*}
 \int_{\mathbb{R}^{3}}\Big((-\Delta)^{\frac{s}{2}}u_{0}(-\Delta)^{\frac{s}{2}}v+V(x)u_{0}v\Big)dx=\int_{\mathbb{R}^{3}}f(x,u_{0})v dx,
\end{align*}
for all $v\in C_{0}^{\infty}(\mathbb{R}^{3})$. So $u_{0}$ is a weak solution of (1.7). From a similar argument of the proof in Lemma 2.4,
 we know that $u_{0}^{\pm}\neq 0$.

\vspace{2mm}

\textbf{Step 3.} The problem (1.7) has a least energy sign-changing solution $v_{0}$, and there is a unique pair $(\alpha_{\lambda_{n}}, \beta_{\lambda_{n}} )\in \mathbb{R}^{+}\times\mathbb{R}^{+}$ such that $\alpha_{\lambda_{n}}{v_{0}}^{+}+ \beta_{\lambda_{n}}{v_{0}}^{-} \in \mathcal{M}_{\lambda}$. Moreover,
 $(\alpha_{\lambda_{n}}, \beta_{\lambda_{n}})\rightarrow (1, 1)$ as $n\rightarrow \infty$.\\

Via a similar argument in the proof of Theorem 1.1, there is a least energy sign-changing solution $v_{0}$ for problem (1.7) with two nodal domain, so we have
\begin{align}
\int_{\mathbb{R}^{3}}\vert(-\Delta)^{\frac{s}{2}}{v_{0}}^{+}\vert^{2} dx + \int_{\mathbb{R}^{3}}(-\Delta)^{\frac{s}{2}}{v_{0}}^{+}(-\Delta)^{\frac{s}{2}}{v_{0}}^{-}dx+\int_{\mathbb{R}^{3}}V(x){{v_{0}}^{+}}^{2}dx
=\int_{\mathbb{R}^{3}}f(x,{v_{0}}^{+}){v_{0}}^{+}dx,
\end{align}
and
\begin{align}
\int_{\mathbb{R}^{3}}\vert(-\Delta)^{\frac{s}{2}}{v_{0}}^{-}\vert^{2} dx + \int_{\mathbb{R}^{3}}(-\Delta)^{\frac{s}{2}}{v_{0}}^{+}(-\Delta)^{\frac{s}{2}}{v_{0}}^{-}dx
+\int_{\mathbb{R}^{3}}V(x){v_{0}^{-}}^{2}dx=\int_{\mathbb{R}^{3}}f(x,{v_{0}}^{-}){v_{0}}^{-}dx.
\end{align}
By Lemma 2.1, there exits an unique pair of $(\alpha_{\lambda_{n}}, \beta_{\lambda_{n}} )$ such that $\alpha_{\lambda_{n}}{v_{0}}^{+}+ \beta_{\lambda_{n}}{v_{0}}^{-} \in \mathcal{M}_{\lambda}$. So we have
\begin{align}
&\alpha_{\lambda_{n}}^{2}\int_{\mathbb{R}^{3}}\vert(-\Delta)^{\frac{s}{2}}{v_{0}}^{+}\vert^{2} dx +\alpha_{\lambda_{n}}\beta_{\lambda_{n}} \int_{\mathbb{R}^{3}}(-\Delta)^{\frac{s}{2}}{v_{0}}^{+}(-\Delta)^{\frac{s}{2}}{v_{0}}^{-}dx+\alpha_{\lambda_{n}}^{2}\int_{\mathbb{R}^{3}}V(x){{v_{0}}^{+}}^{2}dx \nonumber\\
&\quad \quad +\lambda_{n} \alpha_{\lambda_{n}}^{4}\int_{\mathbb{R}^{3}}\phi_{{v_{0}}^{+}}^{t}{{v_{0}}^{+}}^{2}dx+\lambda_{n} \alpha_{\lambda_{n}}^{2}\beta_{\lambda_{n}}^{2}\int_{\mathbb{R}^{3}}\phi_{{v_{0}}^{-}}^{t}{{v_{0}}^{+}}^{2}dx
=\int_{\mathbb{R}^{3}}f(x,\alpha_{\lambda_{n}}{v_{0}}^{+})\alpha_{\lambda_{n}} {v_{0}}^{+}dx,
\end{align}
and
\begin{align}
&\beta_{\lambda_{n}}^{2}\int_{\mathbb{R}^{3}}\vert(-\Delta)^{\frac{s}{2}}{v_{0}}^{-}\vert^{2} dx +\alpha_{\lambda_{n}}\beta_{\lambda_{n}} \int_{\mathbb{R}^{3}}(-\Delta)^{\frac{s}{2}}{v_{0}}^{+}(-\Delta)^{\frac{s}{2}}{v_{0}}^{-}dx
+\beta_{\lambda_{n}}^{2}\int_{\mathbb{R}^{3}}V(x){v_{0}^{-}}^{2}dx\nonumber\\
&\quad \quad+\lambda_{n} \beta_{\lambda_{n}}^{4}\int_{\mathbb{R}^{3}}\phi_{{v_{0}}^{-}}^{t}{{v_{0}}^{-}}^{2}dx+\lambda_{n} \alpha_{\lambda_{n}}^{2}\beta_{\lambda_{n}}^{2}\int_{\mathbb{R}^{3}}\phi_{{v_{0}}^{+}}^{t}{{v_{0}}^{-}}^{2}dx=\int_{\mathbb{R}^{3}}f(x,\beta_{\lambda_{n}} {v_{0}}^{-})\beta_{\lambda_{n}}{v_{0}}^{-}dx.
\end{align}
From $(f_{3})$ and $\lambda_{n}\rightarrow 0$ as $n\rightarrow\infty$, we get that the sequences $\{\alpha_{\lambda_{n}}\}$ and $\{\beta_{\lambda_{n}}\}$
are bounded. Assume that $\alpha_{\lambda_{n}}\rightarrow \alpha_{0}$ and $\beta_{\lambda_{n}}\rightarrow \beta_{0}$ as $n\rightarrow \infty$. From (2.16), (3.4) and (3.5), we have
\begin{align}
&\alpha_{0}^{2}\int_{\mathbb{R}^{3}}\vert(-\Delta)^{\frac{s}{2}}{v_{0}}^{+}\vert^{2} dx +\alpha_{0}\beta_{0} \int_{\mathbb{R}^{3}}(-\Delta)^{\frac{s}{2}}{v_{0}}^{+}(-\Delta)^{\frac{s}{2}}{v_{0}}^{-}dx+\alpha_{0}^{2}\int_{\mathbb{R}^{3}}V(x){{v_{0}}^{+}}^{2}dx \nonumber\\
&\quad\quad =\int_{\mathbb{R}^{3}}f(x,\alpha_{0}{v_{0}}^{+})\alpha_{0} {v_{0}}^{+}dx,
\end{align}
and
\begin{align}
&\beta_{0}^{2}\int_{\mathbb{R}^{3}}\vert(-\Delta)^{\frac{s}{2}}{v_{0}}^{-}\vert^{2} dx +\alpha_{0}\beta_{0} \int_{\mathbb{R}^{3}}(-\Delta)^{\frac{s}{2}}{v_{0}}^{+}(-\Delta)^{\frac{s}{2}}{v_{0}}^{-}dx
+\beta_{0}^{2}\int_{\mathbb{R}^{3}}V(x){v_{0}^{-}}^{2}dx\nonumber\\
&\quad \quad=\int_{\mathbb{R}^{3}}f(x,\beta_{0} {v_{0}}^{-})\beta_{0}{v_{0}}^{-}dx.
\end{align}
Moreover, by $(f_{3})$ and $(f_{4})$, we know that $\frac{f(x, s)}{\vert s\vert^{3}}$ is nondecreasing in $\vert s\vert$. So from (3.2), (3.3),
(3.6), (3.7), we obtain that $(\alpha_{0}, \beta_{0})=(1, 1)$.\\

Now we complete the proof of Theorem 1.3. We only need to show that $u_{0}$ obtained in step 2 is a least energy sign-changing solution of problem
(1.7).
By Lemma 2.3, we have
\begin{align*}
I_{0}(v_{0})\leq I_{0}(u_{0})&\leq\underset{n\rightarrow\infty}{\lim}I_{\lambda_{n}}(u_{\lambda_{n}})
=\underset{n\rightarrow\infty}{\lim}I_{\lambda_{n}}(u_{\lambda_{n}}^{+}+u_{\lambda_{n}}^{-})\\
&\leq \underset{n\rightarrow\infty}{\lim}I_{\lambda_{n}}(\alpha_{\lambda_{n}}{v_{0}}^{+}+\beta_{\lambda_{n}}{v_{0}}^{-})\\
&=I_{0}(v_{0}).
\end{align*}
This show that $u_{0}$ is a least energy sign-changing solution of problem (1.7) which has precisely two
nodal domains. We complete the proof of Theorem 1.3.
\end{proof}

\section*{Acknowledgements}
C. Ji is supported by NSFC (grant No. 11301181) and China Postdoctoral Science Foundation funded project.

\end{document}